\documentclass[runningheads]{llncs}
\usepackage[T1]{fontenc}
\usepackage{graphicx}
\usepackage{color} %subfigure, colortbl
\usepackage{amsmath, amssymb, xparse}%
\usepackage{booktabs}

\usepackage{xfrac}

\usepackage{enumitem}
\setlist[enumerate]{leftmargin=.5in}
\setlist[itemize]{leftmargin=.5in}

\usepackage[
  bookmarks,
  linkcolor= blue,
  citecolor= blue,
  filecolor= blue,
  urlcolor=  blue,
  colorlinks=true,
  hyperindex=true,
  pdftitle={3 Step Am2 model},
  pdfauthor={R. Fekih-Salem}
]{hyperref}

\graphicspath{{Images/}}

\usepackage{color}

\begin{document}
 
\title{Qualitative Analysis of a Three-Stage Anaerobic Digestion Model with Microbial Decay}
\titlerunning{Three-Stage Anaerobic Digestion Model}
\author{Radhouane Fekih-Salem\inst{1}\orcidID{0000-0003-1168-4930}
%\and Second Author\inst{2,3}\orcidID{1111-2222-3333-4444} \and
% Third Author\inst{3}\orcidID{2222--3333-4444-5555}
}
\authorrunning{R. Fekih-Salem}
\institute{University of Tunis El Manar, National Engineering School of Tunis, LAMSIN, 1002, Tunis, Tunisia 
\email{radhouane.fekih-salem@enit.utm.tn}%\\
}
\maketitle              % typeset the header of the contribution

\begin{abstract}
We extend a three-stage anaerobic digestion model by incorporating microbial mortality into the hydrolysis of particulate organic matter. The model describes hydrolysis, acidogenesis, and methanogenesis, each with distinct dilution and decay rates, and accounts for non-monotonic growth in order to capture substrate inhibition. Two hydrolysis mechanisms are considered: a first-order formulation and a biomass-dependent one. The latter distinguishes hydraulic washout from intrinsic mortality and leads to different persistence conditions. For both models, we establish well-posedness and prove the positivity and boundedness of solutions. We then characterize all equilibria and analyze their local stability with respect to key operating parameters.
The inclusion of microbial mortality provides a more general and biologically relevant framework, thereby enriching the qualitative dynamics compared to the classical AM2 model.

\end{abstract}
%%%%%%%%%%%%%%%%%%%%%%%%%%%%%%%%%%%%%%%%%%%%%%%%%%%%%%% 
%%%%%%%%%%%%%%%%%%%%%%%%%%%%%%%%%%%%%%%%%%%%%%%%%%%%%%% 
\section{Introduction}                \label{SecIntro}
%%%%%%%%%%%%%%%%%%%%%%%%%%%%%%%%%%%%%%%%%%%%%%%%%%%%%%% 
%%%%%%%%%%%%%%%%%%%%%%%%%%%%%%%%%%%%%%%%%%%%%%%%%%%%%%% 
Anaerobic digestion (AD) is a biological process used for the treatment of organic wastes, with the additional benefit of producing renewable energy in the form of biogas. During this process, complex organic matter is successively transformed through hydrolysis, acidogenesis, and methanogenesis into methane and carbon dioxide.
Despite its wide applications, AD remains a complex process due to the coupling between biological and physicochemical phenomena and the difficulty of model calibration. To address this complexity, several mathematical models have been proposed, ranging from detailed to reduced formulations \cite{BernardBB2001,SariMB2017}. Reduced models, in particular, aim to capture the essential dynamics while remaining mathematically tractable.

In this context, simplified models involving two or three microbial populations have been widely studied \cite{SariMB2016,SariMB2017}. When each microbial growth rate depends only on its own substrate, the system exhibits a cascade structure and describes a commensalistic relationship, in which one species benefits from another without affecting it \cite{BenyahiaJPC2012,BernardBB2001,FekihArima2014,SariProcesses2022,SariNonLinDyn2021}. In contrast, when the growth of one species depends on an intermediate substrate produced by another, the model captures a syntrophic relationship, in which both populations are interdependent for their survival \cite{NouaouraSIAP2021,NouaouraDcDs2021,SariMB2016,SariMB2017}.
These models exhibit rich dynamical behaviors, including coexistence equilibria, multistability, and oscillations arising from nonlinear interactions. In particular, the emergence of limit cycles through Hopf bifurcations has been reported in \cite{SariMB2017}, and further investigated in \cite{NouaouraSIAP2021,NouaouraDcDs2021} by incorporating additional effects such as inhibition and microbial decay. These results highlight the importance of biological mechanisms and operating conditions in shaping the qualitative behavior of AD systems.

In this work, we revisit this three-stage framework by incorporating microbial mortality, which leads to a more realistic description of population dynamics and significantly affects persistence and stability properties. The originality of this study lies in the combined consideration of mortality effects and non-monotonic growth kinetics, as well as in the comparison of two modeling approaches for the hydrolysis step.
More precisely, we consider a chemostat model describing the sequential transformation of a particulate substrate $X_0$ into a soluble substrate $S_1$ through hydrolysis, followed by its conversion into intermediate products $S_2$ by acidogenic bacteria $X_1$, and finally into biogas by methanogenic bacteria $X_2$. The system operates in a continuous reactor of constant volume with dilution rate $D$.
Two hydrolysis mechanisms are investigated. The first assumes a constant enzymatic activity, leading to first-order kinetics
\[
r_0 = k_{\rm hyd} X_0,
\]
where hydrolysis depends only on the substrate concentration. The second explicitly involves the acidogenic biomass, resulting in biomass-dependent kinetics
\[
r_0 = \mu_0(X_0)\, X_1,
\]
which allows a clear distinction between hydraulic washout and intrinsic mortality, and may induce different persistence thresholds for the microbial populations.

The main objective of this paper is to analyze how the inclusion of microbial mortality and the choice of hydrolysis modeling affect the qualitative behavior of the system. In particular, we study the existence and stability of equilibria and highlight the differences between models with and without an explicit hydrolytic microbial compartment.

The paper is organized as follows. Section \ref{SecModel} introduces the general mathematical model. Sections \ref{SecNoComp} and \ref{SecWithComp} are devoted to the analysis of the models without and with a hydrolytic microbial compartment, respectively. Section \ref{SecConclu} concludes the paper.
Technical proofs related to the well-posedness of the model, as well as the positivity and boundedness of solutions, are given in Appendix \ref{SecAppWPPS}. Further analytical results concerning the existence and multiplicity of equilibria for the model with a microbial hydrolytic compartment, together with their local stability analysis, are provided in Appendices \ref{SecAppBMult} and \ref{SecAppLES}, respectively.
%%%%%%%%%%%%%%%%%%%%%%%%%%%%%%%%%%%%%%%%%%%%%%%%%%%%%%%
%%%%%%%%%%%%%%%%%%%%%%%%%%%%%%%%%%%%%%%%%%%%%%%%%%%%%%% 
\section{Mathematical model}           \label{SecModel}
%%%%%%%%%%%%%%%%%%%%%%%%%%%%%%%%%%%%%%%%%%%%%%%%%%%%%%% 
%%%%%%%%%%%%%%%%%%%%%%%%%%%%%%%%%%%%%%%%%%%%%%%%%%%%%%%
To analyze the effects of microbial mortality and the hydrolysis mechanism on anaerobic digestion, we consider the following chemostat model describing the sequential transformation of substrates by microbial populations:
%=============================================
\begin{equation}           \label{Am2HydModel}
\left\{
\begin{aligned}
\dot{X}_0 &= D X_0^{in} - \alpha_0 D X_0 - r_0, \\
\dot{S}_1 &= D(S_1^{in}-S_1) + k_0 r_0 - k_1 \mu_1(S_1) X_1, \\
\dot{X}_1 &= (\mu_1(S_1) - D_1) X_1, \\
\dot{S}_2 &= D(S_2^{in}-S_2) + k_2 \mu_1(S_1) X_1 - k_3 \mu_2(S_2) X_2, \\
\dot{X}_2 &= (\mu_2(S_2) - D_2) X_2.
\end{aligned}
\right.
\end{equation}
%=============================================
Here, $X_0$ denotes the slowly biodegradable substrate, $S_i$ the soluble substrates, and $X_i$ ($i=1,2$) the microbial populations. The functions $\mu_i$ represent the specific growth rates, $S_i^{in}$ are influent concentrations, and $D$ is the dilution rate.
%=============================================
The effective removal rates of microbial populations account for both hydraulic washout and intrinsic mortality:
%=============================================
\begin{equation}                                          \label{D1D2}
D_i = \alpha_i D + a_i, \qquad i=1,2,
\end{equation}
%=============================================
where $\alpha_i \in [0,1]$ represent fractions subject to washout, and $a_i \ge 0$ are the intrinsic mortality rates. This formulation decouples the hydraulic retention time ${\rm HRT}=1/D$ from the solid retention time ${\rm SRT}=1/(\alpha_i D)$ \cite{BenyahiaJPC2012,SariNonLinDyn2021}.
%=============================================
The coefficients $k_i$, $i=0,\ldots,3$, denote the yield factors, constrained by
\[
1 \ge k_0, \quad k_1 \ge 1+k_2, \quad k_3 \ge 1,
\]
ensuring mass conservation.
%=============================================
Two hydrolysis mechanisms are considered:
\[
r_0 = k_{\rm hyd} X_0 \quad \text{(first-order)}, \qquad
r_0 = \mu_0(X_0) X_1 \quad \text{(biomass-dependent)},
\]
allowing a comparison of models with and without an explicit hydrolytic microbial compartment, and capturing differences between washout and intrinsic mortality effects.
%=============================================
The growth functions satisfy the following general assumptions (continuously differentiable, $\mathcal{C}^1$):
%=============================================
\begin{description}
\item[H1:] $\mu_1(S_1)$ is monotone increasing with $\mu_1(0)=0$.
\item[H2:] $\mu_2(S_2)$ is non-monotone, with $\mu_2(0)=\mu_2(+\infty)=0$; there exists $S_2^m>0$ such that
\[
\mu_2'(S_2)>0 \text{ for } 0<S_2<S_2^m, \quad
\mu_2'(S_2)<0 \text{ for } S_2>S_2^m.
\]
\item[H3:] $\mu_0(0)=0$, $\mu_0$ is increasing and concave on $\mathbb{R}_+$.
\end{description}
%=============================================
Assumption H3 is satisfied, for example, when $\mu_0$ is linear or of Monod type.
%=================
As a preliminary step, we establish that the model is biologically well-posed, in the sense that all concentrations remain nonnegative and bounded over time.
%=================
\begin{proposition}                               \label{PropDefModel}
All solutions of system (\ref{Am2HydModel}) with nonnegative initial conditions remain nonnegative and bounded for all $t\ge 0$. 
Let $Z = k_0 X_0 + S_1 + S_2 + (k_1 - k_2) X_1 + k_3 X_2$
denote the total mass density in the chemostat, and define the corresponding input quantity by $Z^{in} = k_0 X_0^{in} + S_1^{in} + S_2^{in}$. Then, the set
\[
\Omega = \left\{ (X_0,S_1,X_1,S_2,X_2)\in \mathbb{R}_+^5 : Z \le \tfrac{D}{D_{\min}} Z^{in} \right\}
\]
is positively invariant and is a global attractor for system (\ref{Am2HydModel}), where $D_{\min}=\min(\alpha_0 D, D_1, D_2)$.
\end{proposition}
%%%%%%%%%%%%%%%%%%%%%%%%%%%%%%%%%%%%%%%%%%%%%%%%%%%%%%% 
\section{Model without a hydrolytic microbial compartment}   \label{SecNoComp}
%%%%%%%%%%%%%%%%%%%%%%%%%%%%%%%%%%%%%%%%%%%%%%%%%%%%%%% 
We consider the case where hydrolysis follows first-order kinetics
\[
r_0 = k_{\rm hyd} X_0,
\]
corresponding to a model without an explicit hydrolytic microbial compartment. 
In this case, the dynamics of $X_0$ are decoupled from the rest of the system and satisfy a linear equation. Hence, $X_0$ globally converges to
%================================
\begin{equation}     \label{ExpX0*}
X_0^* = \tfrac{D}{k_{\rm hyd}+\alpha_0 D} X_0^{in}.
\end{equation}
%================================
As a consequence, the asymptotic behavior of system (\ref{Am2HydModel}) is governed by the reduced four-dimensional subsystem in $(S_1,X_1,S_2,X_2)$.
At equilibrium $X_0 = X_0^*$, system (\ref{Am2HydModel}) reduces to the following AM2-type model \cite{BenyahiaJPC2012,SariNonLinDyn2021}:
%=============================================
\begin{equation}                \label{AM2Aug}
\left\{
\begin{aligned}
\dot{S}_1 &= D(S_1^{in*}-S_1) - k_1 \mu_1(S_1) X_1,\\
\dot{X}_1 &= (\mu_1(S_1) - D_1) X_1,\\
\dot{S}_2 &= D(S_2^{in}-S_2) + k_2 \mu_1(S_1) X_1 - k_3 \mu_2(S_2) X_2,\\
\dot{X}_2 &= (\mu_2(S_2) - D_2) X_2,
\end{aligned}
\right.
\end{equation}
%=============================================
where the effective input concentration of $S_1$ is
%================================
\begin{equation}     \label{ExpS1in*}
S_1^{in*} = S_1^{in} + \tfrac{k_0 k_{\rm hyd}}{k_{\rm hyd}+\alpha_0 D} X_0^{in}.
\end{equation}
%================================
Each equilibrium $F=(S_1^*,X_1^*,S_2^*,X_2^*)$ of system (\ref{AM2Aug}) uniquely defines an equilibrium 
$E=(X_0^*,S_1^*,X_1^*,S_2^*,X_2^*)$ of the full model (\ref{Am2HydModel}), and their local stability properties coincide. 
%============================================= 
We denote the equilibria of system (\ref{Am2HydModel}) by $E_j^i$, where 
$j\in\{0,1\}$ indicates washout $(j=0)$ or persistence $(j=1)$ of the first microbial population, 
and $i\in\{0,1,2\}$ characterizes the steady states of the second population: $i=0$ corresponds to washout, while $i=1,2$ denote distinct positive equilibria when multiple solutions exist. 
%=============================================
With this notation, system (\ref{Am2HydModel}) admits up to six equilibria, which can be classified as follows:
\begin{itemize}[leftmargin=*]
\item $E_0^0$: washout equilibrium, where both microbial populations are absent.
\item $E_0^i$: equilibria where only the second population persists (possibly with two distinct steady states).
\item $E_1^0$: equilibria where only the first population persists.
\item $E_1^i$: coexistence equilibria, where both populations persist (possibly with two distinct steady states).
\end{itemize}
%=============================================

The equilibria of system (\ref{Am2HydModel}) can be fully characterized using the break-even concentrations and auxiliary functions introduced in Table \ref{tab:auxiliary}. 
The following results extend those obtained in \cite{BenyahiaJPC2012,SariNonLinDyn2021} to the case of distinct removal rates $D_i$, and are derived along similar lines.
%=============================================
\begin{proposition}
Assume that Hypotheses \textbf{H1} and \textbf{H2} hold. Then all equilibria of system (\ref{Am2HydModel}) with 
$r_0 = k_{\rm hyd} X_0$, 
denoted by $E_j^i$ with $j\in\{0,1\}$ and $i\in\{0,1,2\}$, are completely characterized as follows: their components are given in Table~\ref{TableComp3StepFreeHMC}, and their existence and local stability conditions in Table~\ref{TableExisStab3StepFreeHMC}.
\end{proposition}
%=============================================
\begin{table}[!ht]
\centering
\setlength{\abovecaptionskip}{3pt}
\setlength{\belowcaptionskip}{3pt}
\caption{Break-even concentrations and auxiliary functions.}\label{tab:auxiliary}
\begin{tabular}{l}
\toprule
$\lambda_1(D_1)$: unique solution of $\mu_1(S_1)=D_1$ \\
\midrule
$\lambda_2^i(D_2),\ i=1,2$: solutions of $\mu_2(S_2)=D_2$ for $D_2<\mu_2(S_2^m)$; \\
$\lambda_2^1(0)=0$, $\lambda_2^2(0)=+\infty$, and $\lambda_2^1=\lambda_2^2$ if $D_2=\mu_2(S_2^m)$ \\
\midrule
$H_i(D)=\lambda_2^i(D_2)+\tfrac{k_2}{k_1}\lambda_1(D_1),\quad i=1,2$ \\
\bottomrule
\end{tabular}
\end{table}
%=============================================
\begin{table}[!ht]
\setlength{\abovecaptionskip}{3pt}
\setlength{\belowcaptionskip}{3pt}
\caption{Components of all equilibria of model (\ref{Am2HydModel}). The quantities $X_0^*$ and $S_1^{in*}$ are defined in \eqref{ExpX0*} and \eqref{ExpS1in*}. Auxiliary functions are given in Table~\ref{tab:auxiliary}.}
\label{TableComp3StepFreeHMC}
\begin{center}
\renewcommand{\arraystretch}{1.2}
\begin{tabular}{l c c c c c}
\toprule
 & $X_0^*$ & $S_1^*$ & $X_1^*$ & $S_2^*$ & $X_2^*$ \\
\midrule
$E_0^0$ 
& $X_0^*$ 
& $S_1^{in*}$ 
& $0$ 
& $S_2^{in}$ 
& $0$ \\
$E_0^i$ 
& $X_0^*$ 
& $S_1^{in*}$ 
& $0$ 
& $\lambda_2^i$ 
& $\tfrac{D}{k_3 D_2}(S_2^{in}-\lambda_2^i)$ \\
$E_1^0$ 
& $X_0^*$ 
& $\lambda_1$ 
& $\tfrac{D}{k_1 D_1}(S_1^{in*}-\lambda_1)$ 
& $S_2^{in}+\tfrac{k_2}{k_1}(S_1^{in*}-\lambda_1)$ 
& $0$ \\
$E_1^i$ 
& $X_0^*$ 
& $\lambda_1$ 
& $\tfrac{D}{k_1 D_1}(S_1^{in*}-\lambda_1)$ 
& $\lambda_2^i$ 
& $\tfrac{D}{k_3 D_2}\left(S_2^{in}+\tfrac{k_2}{k_1}S_1^{in*}-H_i(D)\right)$ \\
\bottomrule
\end{tabular}
\end{center}
\end{table}
%\vspace{-0.8cm} 
%=============================================
\begin{table}[!ht]
\setlength{\abovecaptionskip}{3pt}
\setlength{\belowcaptionskip}{3pt}
\caption{Existence and stability conditions of equilibria of model (\ref{Am2HydModel}). The quantities $X_0^*$ and $S_1^{in*}$ are defined in \eqref{ExpX0*} and \eqref{ExpS1in*}. Auxiliary functions are given in Table~\ref{tab:auxiliary}.}
\label{TableExisStab3StepFreeHMC}
%=============================================
\begin{center}
\renewcommand{\arraystretch}{1.2}
\begin{tabular}{ @{\hspace{1mm}}l@{\hspace{2mm}}  @{\hspace{2mm}}l@{\hspace{2mm}} @{\hspace{2mm}}l@{\hspace{1mm}} }	
\toprule
 & Existence & Stability \\
\midrule
$E_0^0$ 
& always 
&
$\left\{\begin{tabular}{l}
$S_1^{in} < \lambda_1(D_1)$  \\
$S_2^{in} \notin [\lambda_2^1(D_2),\lambda_2^2(D_2)]$
\end{tabular}
\right.$
\\
$E_0^i$ 
& $S_2^{in}>\lambda_2^i(D_2)$ 
& 
$\left\{\begin{tabular}{ll}
$S_1^{in} < \lambda_1(D_1), \ i=1$\\
$\text{unstable}, \ i=2$
\end{tabular}
\right.$
\\
$E_1^0$ 
& $S_1^{in*}>\lambda_1(D_1)$ 
& $S_2^{in}+\tfrac{k_2}{k_1}S_1^{in*}\notin[H_1(D),H_2(D)]$ \\
$E_1^i$ 
& 
$\left\{\begin{tabular}{l}
$S_1^{in*}   >\lambda_1(D_1)$\\
$S_2^{in}+\tfrac{k_2}{k_1}S_1^{in*}  >H_i(D)$
\end{tabular}
\right.$
& 
$\left\{\begin{tabular}{l}
$\text{stable if it exists},   i=1$\\
$\text{unstable if it exists},   i=2$
\end{tabular}
\right.$
\\
\bottomrule
\end{tabular}
\end{center}
\end{table}
%=============================================

Recall that the persistence condition for species $X_1$ in the AM2 model is $\lambda_1 < S_1^{in}$ \cite{BenyahiaJPC2012,SariNonLinDyn2021}. 
Due to hydrolysis, the effective input increases to $S_1^{in*}$, so that the persistence condition becomes $\lambda_1 < S_1^{in*}$,
which favors the survival of $X_1$ through the additional substrate produced by hydrolysis.
%%%%%%%%%%%%%%%%%%%%%%%%%%%%%%%%%%%%%%%%%%%%%%%%%%%%%%% 
%%%%%%%%%%%%%%%%%%%%%%%%%%%%%%%%%%%%%%%%%%%%%%%%%%%%%%% 
\section{Model with hydrolytic microbial compartment} \label{SecWithComp}
%%%%%%%%%%%%%%%%%%%%%%%%%%%%%%%%%%%%%%%%%%%%%%%%%%%%%%% 
%%%%%%%%%%%%%%%%%%%%%%%%%%%%%%%%%%%%%%%%%%%%%%%%%%%%%%% 
We consider the full hydrolysis-augmented chemostat model (\ref{Am2HydModel}), a five-dimensional system with variables $(X_0, S_1, X_1, S_2, X_2)$, where the hydrolytic microbial compartment follows a general growth function $\mu_0(X_0)$. 
Unlike previous approaches that decouple the first three equations from the remaining ones \cite{FekihArima2014}, we directly analyze the full system, whose complexity remains tractable. The model reads
%=============================================
\begin{equation}          \label{ModelWithHMC}
\left\{
\begin{aligned}
\dot X_0 &= D(X_0^{in}-\alpha_0 X_0) - \mu_0(X_0) X_1, \\
\dot S_1 &= D(S_1^{in}-S_1) + k_0 \mu_0(X_0) X_1 - k_1 \mu_1(S_1) X_1, \\
\dot X_1 &= (\mu_1(S_1) - D_1) X_1, \\
\dot S_2 &= D(S_2^{in}-S_2) + k_2 \mu_1(S_1) X_1 - k_3 \mu_2(S_2) X_2, \\
\dot X_2 &= (\mu_2(S_2) - D_2) X_2.
\end{aligned}
\right.
\end{equation}
%=============================================
Here, $S_1$ and $S_2$ denote the substrates, while $X_0$, $X_1$, and $X_2$ represent the hydrolytic, acidogenic (first), and methanogenic (second) microbial populations, respectively. We assume Hypotheses {\bf H1--H3} on the growth functions.  
%%%%%%%%%%%%%%%%%%%%%%%%%%%%%%%%%%%%%%%%%%%%%%%%%%%%%%% 
\subsection{Existence of equilibria}
%%%%%%%%%%%%%%%%%%%%%%%%%%%%%%%%%%%%%%%%%%%%%%%%%%%%%%% 
Equilibria are obtained by setting the right-hand sides of (\ref{ModelWithHMC}) to zero:
%=============================================
\begin{equation}        \label{SysAlgWithHMC}
\begin{cases}
0 = D(X_0^{in}-\alpha_0 X_0) - \mu_0(X_0) X_1,\\
0 = D(S_1^{in}-S_1) + k_0 \mu_0(X_0) X_1 - k_1 \mu_1(S_1) X_1,\\
0 = (\mu_1(S_1)-D_1) X_1,\\
0 = D(S_2^{in}-S_2) + k_2 \mu_1(S_1) X_1 - k_3 \mu_2(S_2) X_2,\\
0 = (\mu_2(S_2)-D_2) X_2.
\end{cases}
\end{equation}
%=============================================
\begin{itemize}[leftmargin=*,label=\raisebox{0.25ex}{\tiny$\bullet$}]
%=============================================
\item \textbf{Equilibrium $E_0^0$ (total washout).}  
Here, $X_1 = X_2 = 0$. Substituting into the first equation of (\ref{SysAlgWithHMC}) yields $X_0 = X_0^* := X_0^{in}/\alpha_0$
and the second equation gives $S_1 = S_1^{in}$. The fourth equation yields $S_2 = S_2^{in}$. This equilibrium always exists.
%=============================================
\item \textbf{Equilibria $E_0^i$ (only methanogenic population persists).}  
Here, $X_1 = 0$ and $X_2 > 0$. As in the previous case, substituting $X_1=0$ into the first equation gives $X_0 = X_0^*$, and the second equation gives $S_1 = S_1^{in}$. 
From the fourth and fifth equations, using the definition of $\lambda_2^i(D_2)$ (Table \ref{tab:auxiliary}), we obtain
\[
S_2 = \lambda_2^i(D_2), \quad 
X_2 = \tfrac{D}{k_3 D_2}\big(S_2^{in}-\lambda_2^i(D_2)\big), \qquad i=1,2.
\]
The equilibrium exists if and only if $X_2>0$, i.e.,
$S_2^{in}>\lambda_2^i(D_2)$.
%=============================================
\item \textbf{Equilibria $E_1^0$ and $E_1^i$ (acidogenic population persists).}  
Here, $X_1>0$. From the definition of $\lambda_1(D_1)$ (Table \ref{tab:auxiliary}), the third equation of (\ref{SysAlgWithHMC}) yields
\[
S_1 = \lambda_1(D_1).
\]
Substituting into the first equation gives
\[
X_1 = \xi(X_0), \quad \text{with} \quad 
\xi(X_0) = \tfrac{D(X_0^{in}-\alpha_0 X_0)}{\mu_0(X_0)}.
\]
Combining with the second equation leads to
%=============================================
\begin{equation}             \label{ExpDelta}
X_1 = \delta(X_0), \quad 
\delta(X_0) = \tfrac{D}{k_1 D_1} \Big[(S_1^{in}-\lambda_1(D_1)) + k_0 (X_0^{in}-\alpha_0 X_0)\Big].
\end{equation}
%=============================================
Hence, admissible values of $X_0$, denoted by $X_0^k$, are solutions of
\[
\xi(X_0)=\delta(X_0).
\]
The existence and multiplicity of these solutions are completely characterized in Proposition \ref{PropMultiComp} (Appendix \ref{SecAppBMult}). For each solution $X_0^k$, define $X_1^{k*}=\delta(X_0^k)$ and the corresponding equilibria $E_1^{0k}$ and $E_1^{ik}$.
%=============================================
\begin{itemize}[leftmargin=*,label=\raisebox{0.25ex}{\tiny$\bullet$}]
%=============================================
\item \textbf{$E_1^{0k}$ (no methanogenic population).}  
Here $X_2=0$. From the fourth equation of (\ref{SysAlgWithHMC}),
\[
S_2^{k*} = S_2^{in} + k_2 \tfrac{D_1}{D} X_1^{k*} > 0.
\]
Thus, $E_1^{0k}$ exists if and only if $\xi(X_0)=\delta(X_0)$ admits at least one solution.
%=============================================
\item \textbf{$E_1^{ik}$ (coexistence equilibria).}  
Here $X_2>0$. From the fifth equation,
\[
S_2 = \lambda_2^i(D_2),
\]
and substituting into the fourth equation yields
\[
X_2^{ik*} = \tfrac{1}{k_3 D_2} \Big[ D\big(S_2^{in} - \lambda_2^i(D_2)\big) + k_2 D_1 X_1^{k*} \Big].
\]
Hence, $E_1^{ik}$ ($i=1,2$) exists if and only if a solution $X_0^k$ exists and
\[
S_2^{k*} > \lambda_2^i(D_2).
\]
%=============================================
\end{itemize}
%=============================================
\end{itemize}
%=============================================

The hydrolysis-augmented model (\ref{ModelWithHMC}) accounts for a biomass-dependent hydrolysis rate $r_0=\mu_0(X_0)X_1$, coupling the dynamics of the hydrolytic substrate $X_0$ with the acidogenic population $X_1$. This interaction modifies the effective substrate availability and, consequently, the persistence conditions of both microbial populations.
Each equilibrium is denoted by $E_j^{\,i}$, where $j\in\{0,1\}$ indicates the absence ($j=0$) or presence ($j=1$) of the acidogenic population $X_1$, while $i\in\{0,1,2\}$ reflects the possible multiplicity of equilibria associated with the methanogenic population $X_2$.
The following proposition summarizes all equilibria of system (\ref{ModelWithHMC}), together with their components and the corresponding existence and local stability conditions.
The proof of the stability results is given in Appendix B.1.
%=============================================
\begin{proposition}     \label{PropExiStabWMC}
Assume that Hypotheses \textbf{H1}--\textbf{H3} hold. 
Then, all equilibria of system (\ref{ModelWithHMC}), denoted by $E_j^i$ with $j\in\{0,1\}$ and $i\in\{0,1,2\}$, are completely characterized as follows: their components are given in Table \ref{TabCompWithMC}, while their existence and local stability conditions are summarized in Table \ref{TableExisStabWithMC}.
Moreover, the number of equilibria of type $E_1^{0k}$ and $E_1^{ik}$ is equal to the integer
$
N \in \{0,1,2\},
$
defined in \eqref{NumberN} as the number of solutions of
\[
\xi(X_0)=\delta(X_0)
\quad \text{in} \quad \left(0,\sfrac{X_0^{in}}{\alpha_0}\right).
\]
These equilibria are indexed by $k=1,\dots,N$ (and do not exist if $N=0$).
\end{proposition}
%=============================================
%=============================================
\begin{table}[!ht]
\setlength{\abovecaptionskip}{3pt}
\setlength{\belowcaptionskip}{3pt}
\caption{Components of all equilibria of model (\ref{ModelWithHMC}). Functions $\lambda_1(D_1)$, $\lambda_2^i(D_2)$, and $H_i(D)$ are defined in Table \ref{tab:auxiliary}. For equilibria involving $X_1>0$, $X_0^k$ denotes the $k$-th solution of $\xi(X_0)=\delta(X_0)$, with $k=1,\dots,N$, where $N$ is defined in \eqref{NumberN}. 
Moreover, $X_1^{k*}=\xi(X_0^k)$ and $S_2^{k*}=S_2^{in} + k_2 \tfrac{D_1}{D} X_1^{k*}$.}\label{TabCompWithMC}
\begin{center}
\renewcommand{\arraystretch}{1.3}
\begin{tabular}{ @{\hspace{1mm}}l@{\hspace{2mm}}  @{\hspace{2mm}}l@{\hspace{2mm}} @{\hspace{2mm}}l@{\hspace{1mm}} 
@{\hspace{1mm}}l@{\hspace{2mm}}  @{\hspace{2mm}}l@{\hspace{2mm}} @{\hspace{2mm}}l@{\hspace{1mm}}}	
\toprule
 & $X_0^*$ & $S_1^*$ & $X_1^*$ & $S_2^*$ & $X_2^*$ \\
\midrule
$E_0^0$ 
& $\sfrac{X_0^{in}}{\alpha_0}$ 
& $S_1^{in}$ 
& $0$ 
& $S_2^{in}$ 
& $0$
\\
$E_0^i$ 
& $\sfrac{X_0^{in}}{\alpha_0}$ 
& $S_1^{in}$ 
& $0$ 
& $\lambda_2^i(D_2)$ 
& $\tfrac{D}{k_3 D_2}\big(S_2^{in}-\lambda_2^i(D_2)\big)$
\\
$E_1^{0k}$ 
& $X_0^k$ 
& $\lambda_1(D_1)$ 
& $X_1^{k*}$ 
& $S_2^{k*}$ 
& $0$
\\
$E_1^{ik}$ 
& $X_0^k$ 
& $\lambda_1(D_1)$ 
& $X_1^{k*}$ 
& $\lambda_2^i(D_2)$ 
& $\tfrac{D}{k_3 D_2}\left(S_2^{k*}-\lambda_2^i(D_2)\right)$ \\
\bottomrule
\end{tabular}
\end{center}
\end{table}
%=============================================
\begin{table}[!ht]
\setlength{\abovecaptionskip}{3pt}
\setlength{\belowcaptionskip}{3pt}
\caption{Existence and local stability conditions of equilibria of model (\ref{ModelWithHMC}) with hydrolytic microbial compartment. 
The integer $N\in\{0,1,2\}$ denotes the number of solutions of $\xi(X_0)=\delta(X_0)$ in $\left(0,\tfrac{X_0^{in}}{\alpha_0}\right)$, defined in \eqref{NumberN}. 
Moreover, $S_2^{k*}=S_2^{in} + k_2 \tfrac{D_1}{D} X_1^{k*}$ and $c_4$ is defined in \eqref{ExpC4}.}
\label{TableExisStabWithMC}
\begin{center}
\renewcommand{\arraystretch}{1.2}
\begin{tabular}{ @{\hspace{1mm}}l@{\hspace{1mm}}  @{\hspace{1mm}}l@{\hspace{1mm}} @{\hspace{1mm}}l@{\hspace{1mm}} }	
\toprule
          & Existence condition & Stability condition\\
\midrule
$E_0^0$ 
& always 
& $\mu_1(S_1^{in}) < D_1$ and $\mu_2(S_2^{in}) < D_2$ 
\\
$E_0^i$ 
& $S_2^{in} > \lambda_2^i(D_2)$ 
& $\mu_1(S_1^{in}) < D_1$ and $i=1$
\\
$E_1^{0k}$ 
& $N \ge 1$ 
& $\xi'(X_0^k) < -\tfrac{\alpha_0 D k_0}{k_1 D_1}$, $c_4>0$, $S_2^{k*} \notin 
\left[\lambda_2^1(D_2),\,\lambda_2^2(D_2)\right]$
\\
$E_1^{ik}$ 
& $N \ge 1$ and $S_2^{k*} > \lambda_2^i(D_2)$ 
& $\xi'(X_0^k) < -\tfrac{\alpha_0 D k_0}{k_1 D_1}$ and $c_4>0$ \\
\bottomrule
\end{tabular}
\end{center}
\end{table}
%=============================================
\begin{remark}
From \eqref{ExpC4}, we observe that the condition $\alpha_0 D \ge D_1$ ensures $c_4>0$ for the three-dimensional subsystem $(X_0,S_1,X_1)$. In this case, the stability of coexistence equilibria of this subsystem is fully determined by the slope condition
\[
\xi'(X_0^k) < -\tfrac{\alpha_0 D k_0}{k_1 D_1}.
\]
In the limiting case $\alpha_0 D = D_1$, the model reduces to a previously studied configuration (see \cite{FekihArima2014}), where $c_4$ is always positive and coexistence equilibria are stable whenever they exist.

More generally, the sign of $c_4$ is not fully characterized and may vary with parameters, leaving open the possibility of a loss of stability through a Hopf bifurcation.

From a modeling viewpoint, this analysis highlights that in the context of a three-stage anaerobic digestion model with distinct removal (or dilution) rates, the introduction of microbial mortality and compartmental structure significantly enriches the dynamics compared to classical chemostat-type models, and may destabilize coexistence equilibria that are otherwise stable.
\end{remark}
%%%%%%%%%%%%%%%%%%%%%%%%%%%%%%%%%%%%%%%%%%%%%%%%%%%%%%% 
%%%%%%%%%%%%%%%%%%%%%%%%%%%%%%%%%%%%%%%%%%%%%%%%%%%%%%%
\section{Conclusion}                  \label{SecConclu}
%%%%%%%%%%%%%%%%%%%%%%%%%%%%%%%%%%%%%%%%%%%%%%%%%%%%%%%
%%%%%%%%%%%%%%%%%%%%%%%%%%%%%%%%%%%%%%%%%%%%%%%%%%%%%%%
In this work, we investigated a three-stage anaerobic digestion model incorporating microbial mortality and non-monotonic growth kinetics. Two hydrolysis mechanisms were considered: a first-order formulation and a biomass-dependent one, allowing for a comparison between models without and with an explicit hydrolytic microbial compartment.

We established well-posedness and proved the positivity and boundedness of solutions. The analysis of equilibria and their stability shows that microbial mortality, together with distinct removal (dilution) rates, strongly influences the qualitative behavior of the system. In particular, these factors affect the persistence conditions of microbial populations and may destabilize coexistence equilibria that are otherwise stable in the classical AM2 model.

Introducing a hydrolytic microbial compartment leads to a richer and more realistic description of the process, highlighting the interplay between hydrolysis, substrate conversion, and microbial interactions. The stability of coexistence equilibria is governed by a combination of slope conditions on $\xi(X_0)$ and additional requirements, such as the sign condition $c_4>0$ for the three-dimensional subsystem, together with admissibility conditions for the $(S_2,X_2)$ dynamics. A complete characterization of these conditions remains an open problem. This leaves open the possibility of qualitative changes in the dynamics, including the loss of stability of coexistence equilibria and the emergence of oscillatory regimes via Hopf bifurcations for suitable parameter values.

Overall, these results emphasize that biological parameters, such as microbial mortality, together with operational parameters, such as distinct dilution rates across compartments, play a crucial role in shaping the dynamics of anaerobic digestion systems. Future work will focus on a detailed bifurcation analysis, on the identification of parameter regimes leading to oscillatory behavior, as well as on optimization strategies for improved process performance and control.
%%%%%%%%%%%%%%%%%%%%%%%%%%%%%%%%%%%%%%%%%%%%%%%%%%%%%%% 
%%%%%%%%%%%%%%%%%%%%%%%%%%%%%%%%%%%%%%%%%%%%%%%%%%%%%%% 
\appendix
%%%%%%%%%%%%%%%%%%%%%%%%%%%%%%%%%%%%%%%%%%%%%%%%%%%%%%% 
%%%%%%%%%%%%%%%%%%%%%%%%%%%%%%%%%%%%%%%%%%%%%%%%%%%%%%% 
\section{Proofs}                    \label{SecAppAProof}
%%%%%%%%%%%%%%%%%%%%%%%%%%%%%%%%%%%%%%%%%%%%%%%%%%%%%%% 
%%%%%%%%%%%%%%%%%%%%%%%%%%%%%%%%%%%%%%%%%%%%%%%%%%%%%%% 
\subsection{Well-posedness and Positivity of Solutions} \label{SecAppWPPS}
%%%%%%%%%%%%%%%%%%%%%%%%%%%%%%%%%%%%%%%%%%%%%%%%%%%%%%% 
\begin{proof}[Proposition \ref{PropDefModel}]
%=============================================
%=============================================
Since the vector field of system (\ref{Am2HydModel}) is $\mathcal{C}^1$, existence and uniqueness of solutions follow. For $i=1,2$,
\[
X_i(\tau)=0 \quad \text{for some } \tau\ge0 \quad \Rightarrow \quad \dot X_i(\tau)=0 .
\]
Hence, if $X_i(0)=0$, then $X_i(t)=0$ for all $t\ge0$, since the boundary face $X_i\equiv0$ is invariant for the vector field defined by (\ref{Am2HydModel}).  
If $X_i(0)>0$, then $X_i(t)>0$ for all $t\ge0$, since trajectories starting with $X_i(0)>0$ cannot reach the boundary $X_i=0$ in finite time by uniqueness of solutions.
Similarly, for the other components:
%=============================================
\begin{align*}
X_0(\tau)=0, & \quad \Rightarrow \quad \dot X_0(\tau)=D X_0^{in},\\
S_1(\tau)=0, & \quad \Rightarrow \quad \dot S_1(\tau)=D S_1^{in}+k_0 r_0(\tau),\\
S_2(\tau)=0, & \quad \Rightarrow \quad \dot S_2(\tau)=D S_2^{in}+k_2\mu_1(S_1(\tau))X_1(\tau).
\end{align*}
%=============================================
Therefore, if $X_0(0)\ge0$ and $S_i(0)\ge0$, $i=1,2$, then $X_0(t)\ge0$ and $S_i(t)\ge0$ for all $t\ge0$.  
Indeed, assume for example that $X_0(0)\ge0$ and that there exists $\tau>0$ such that $X_0(\tau)=0$ while $X_0(t)>0$ for all $t\in(0,\tau)$. 
Then necessarily $\dot X_0(\tau)\le0$, which contradicts $\dot X_0(\tau)=D X_0^{in}>0$.
From system (\ref{Am2HydModel}), we obtain
%=============================================
\begin{align*}
\dot Z
&= D Z^{in} - \alpha_0 D k_0 X_0 - D S_1 - D_1 (k_1-k_2)X_1 - D S_2 - D_2 k_3 X_2 \\
&\le D Z^{in} - D_{\min} Z = -D_{\min}\!\left(Z-\tfrac{D}{D_{\min}}Z^{in}\right),
\end{align*}
%=============================================
where $D_{\min}=\min(\alpha_0 D,D_1,D_2)$. Applying Gronwall's lemma yields
%=============================================
\begin{equation}   \label{Ineg-Z}
Z(t)\le
\tfrac{D}{D_{\min}}Z^{in}
+\Big(Z(0)-\tfrac{D}{D_{\min}}Z^{in}\Big)e^{-D_{\min}t},
\qquad t\ge0 .
\end{equation}
%=============================================
Consequently,
\[
Z(t)\le
\max\!\left(Z(0),\,\tfrac{D}{D_{\min}}Z^{in}\right),
\qquad t\ge0 .
\]
Hence, all solutions remain bounded for all $t\ge0$.  
Moreover, inequality (\ref{Ineg-Z}) implies that the set $\Omega$ is positively invariant and constitutes a global attractor for system (\ref{Am2HydModel}).
\end{proof}
%=============================================
%%%%%%%%%%%%%%%%%%%%%%%%%%%%%%%%%%%%%%%%%%%%%%%%%%%%%%%%
\subsection{Multiplicity of equilibria of model (\ref{AM2Aug}).}    \label{SecAppBMult}
%%%%%%%%%%%%%%%%%%%%%%%%%%%%%%%%%%%%%%%%%%%%%%%%%%%%%%%%
%%%%%%%%%%%%%%%%%%%%%%%%%%%%%%%%%%%%%%%%%%%%%%%%%%%%%%%%
In this section, we analyze the multiplicity of equilibria with $X_1>0$.
To this end, we study the solutions of the scalar equation
$\xi(X_0)=\delta(X_0)$, which characterizes the admissible values of $X_0$.
%=============================================
\begin{lemma}                    \label{lemXi}
Under assumption {\bf H3}, the function $\xi(\cdot)$ vanishes at 
$X_0=\sfrac{X_0^{in}}{\alpha_0}$, is decreasing and convex on 
$\left(0,\sfrac{X_0^{in}}{\alpha_0}\right)$.
\end{lemma}
%=============================================
\begin{proof}
For any $X_0\in\left(0,\sfrac{X_0^{in}}{\alpha_0}\right)$, we have
%=============================================
\begin{equation}               \label{XiPrime}
\xi'(X_0)=-\frac{\alpha_0 D}{\mu_0(X_0)}-\frac{\xi(X_0)}{\mu_0(X_0)}\mu_0'(X_0)<0.
\end{equation}
%=============================================
Moreover,
\[
\xi''(X_0)=-\frac{\xi(X_0)}{\mu_0(X_0)}\mu_0''(X_0)-2\frac{\mu_0'(X_0)}{\mu_0(X_0)}\xi'(X_0).
\]
Under assumption {\bf H3}, we have $\mu_0''(X_0)\le 0$ and $\xi'(X_0)<0$,
which implies $\xi''(X_0)>0$. Therefore, $\xi(\cdot)$ is decreasing and convex on 
$\left(0,\sfrac{X_0^{in}}{\alpha_0}\right)$.
\end{proof}
%=============================================
\begin{lemma}                   \label{lemXi2}
The equation 
%=============================================
\begin{equation}            \label{EqXipEqRap}
\xi'(X_0)=-\tfrac{\alpha_0 D k_0}{k_1 D_1}
\end{equation}
%=============================================
admits a unique solution 
$\bar X_0 \in \left(0,\sfrac{X_0^{in}}{\alpha_0}\right)$
if and only if
%=============================================
\begin{equation}            \label{EquivXipMu0}
\xi'\!\left(\sfrac{X_0^{in}}{\alpha_0}\right) > -\tfrac{\alpha_0 D k_0}{k_1 D_1}
\quad \Longleftrightarrow \quad
k_0 \, \mu_0\!\left(\sfrac{X_0^{in}}{\alpha_0}\right) > k_1 D_1.
\end{equation}
%=============================================
\end{lemma}
%=============================================
%=============================================
\begin{proof}
From Lemma \ref{lemXi}, the function $\xi'(\cdot)$ is  increasing on 
$\left(0,\sfrac{X_0^{in}}{\alpha_0}\right)$. 
Moreover, from \eqref{XiPrime}, we have
\[
\lim_{X_0\to 0} \xi'(X_0) = -\infty.
\]
Hence, the equation \eqref{EqXipEqRap} admits a unique solution 
$\bar X_0 \in \left(0,\sfrac{X_0^{in}}{\alpha_0}\right)$
if and only if the left-hand side of \eqref{EquivXipMu0} holds.
Furthermore, using \eqref{XiPrime}, we obtain
\[
\xi'\!\left(\sfrac{X_0^{in}}{\alpha_0}\right) 
= -\frac{\alpha_0 D}{\mu_0\!\left(\sfrac{X_0^{in}}{\alpha_0}\right)}.
\]
Hence, we obtain the right-hand side of \eqref{EquivXipMu0}.
\end{proof}
%=============================================
%=============================================
\begin{proposition}[Multiplicity of solutions] \label{PropMultiComp}
%=============================================
%=============================================
Assume that {\bf H1}--{\bf H3} hold. 
If the condition \eqref{EquivXipMu0} holds, then there exists a threshold $\bar{S}_1^{in}$ defined by 
%=============================================
\begin{equation}            \label{ExpS1inbar}
\bar{S}_1^{in} = \lambda_1(D_1) + \tfrac{k_1 D_1}{D}\,\bar X_1
- k_0 \big(X_0^{in}-\alpha_0 \bar X_0\big),
\end{equation}
%=============================================
where $\bar X_0$ is the unique solution of equation \eqref{EqXipEqRap} and $\bar X_1 = \delta(\bar X_0)$.
Let $N$ denote the number of solutions of $\xi(X_0)=\delta(X_0)$ in $\left(0,\sfrac{X_0^{in}}{\alpha_0}\right)$. It is given by
%=============================================
\begin{equation}                      \label{NumberN}
N =
\begin{cases}
1, & \text{if } S_1^{in} > \lambda_1(D_1) 
\ \text{and } k_0 \mu_0\!\left(\sfrac{X_0^{in}}{\alpha_0}\right) \le k_1 D_1, \\
0, & \text{if } S_1^{in} \le \lambda_1(D_1) 
\ \text{and } k_0 \mu_0\!\left(\sfrac{X_0^{in}}{\alpha_0}\right) \le k_1 D_1, \\
1, & \text{if } S_1^{in} \ge \lambda_1(D_1) 
\ \text{and } \eqref{EquivXipMu0} \text{ holds}, \\
2, & \text{if } \max(0,\bar S_1^{in}) < S_1^{in} < \lambda_1(D_1) 
\ \text{and } \eqref{EquivXipMu0} \text{ holds}, \\
0, & \text{if } S_1^{in} \le \max(0,\bar S_1^{in}) 
\ \text{and } \eqref{EquivXipMu0} \text{ holds}.
\end{cases}
\end{equation}
%=============================================
\end{proposition}
%=============================================
\begin{proof}
From Lemma \ref{lemXi2}, condition \eqref{EquivXipMu0} is equivalent to the existence of a unique 
$\bar X_0 \in \left(0,\sfrac{X_0^{in}}{\alpha_0}\right)$ solution to \eqref{EqXipEqRap}.
Setting $\bar X_1 = \delta(\bar X_0)$ (see \eqref{ExpDelta}), we define the threshold
$\bar S_1^{in}$ as the value of $S_1^{in}$ for which $\bar X_0$ is a solution of
$\xi(X_0)=\delta(X_0)$. Using \eqref{ExpDelta}, we obtain
\[
\bar X_1
=
\tfrac{D}{k_1 D_1}
\Big[
k_0 (X_0^{in}-\alpha_0 \bar X_0)
+
(\bar S_1^{in}-\lambda_1(D_1))
\Big],
\]
which leads to \eqref{ExpS1inbar}.
We now study the number of solutions of the equation $\xi(X_0)=\delta(X_0)$
on $\left(0,\sfrac{X_0^{in}}{\alpha_0}\right)$.
From Lemma \ref{lemXi}, the function $\xi$ is decreasing and convex, while $\delta$ is affine. Hence, the number of intersections between the graphs of $\xi$ and $\delta$ is determined by their relative position.\\
%=============================================
\noindent \textbf{Case 1:} 
$k_0 \, \mu_0\!\left(\sfrac{X_0^{in}}{\alpha_0}\right) \le k_1 D_1$.

Define the function 
\[
\phi(X_0) = \xi(X_0) - \delta(X_0), \quad X_0 \in \left(0,\sfrac{X_0^{in}}{\alpha_0}\right).
\]

From \eqref{EquivXipMu0} we have 
\[
\xi'\!\left(\sfrac{X_0^{in}}{\alpha_0}\right) \le -\frac{\alpha_0 D k_0}{k_1 D_1}.
\]
Since $\xi'$ is increasing on $(0, X_0^{in}/\alpha_0)$ by Lemma~\ref{lemXi}, it follows that
\[
\xi'(X_0) \le -\frac{\alpha_0 D k_0}{k_1 D_1} = \delta'(X_0), \quad \text{for all} \quad X_0 \in \left(0,\sfrac{X_0^{in}}{\alpha_0}\right),
\]
so that $\phi'(X_0) = \xi'(X_0) - \delta'(X_0) \le 0$, i.e., $\phi$ is non-increasing on the interval.
Moreover, we have
\[
\lim_{X_0 \to 0} \phi(X_0) = +\infty, \qquad
\phi\!\left(\sfrac{X_0^{in}}{\alpha_0}\right) = -\,\frac{D}{k_1 D_1}\big(S_1^{in}-\lambda_1(D_1)\big).
\]
Therefore, $\phi(X_0)=0$ admits a solution if and only if $S_1^{in} > \lambda_1(D_1)$,
and in that case the solution is unique due to the monotonicity of \(\phi\).
The corresponding values of $N$ are given by the first two lines of \eqref{NumberN}.
\medskip

%=============================================
%=============================================
\noindent \textbf{Case 2:} $k_0 \, \mu_0\!\left(\sfrac{X_0^{in}}{\alpha_0}\right) > k_1 D_1$.

The proof can be carried out using arguments analogous to those of Case 1, relying on the convexity of $\phi$ and a standard geometric analysis of the intersections between $\xi$ and $\delta$. For brevity, the details are omitted.
\end{proof}
%=============================================
%=============================================
\subsection{Local stability of equilibria of model (\ref{AM2Aug}).} \label{SecAppLES}
%=============================================
%=============================================
In this section, we study the local stability of the equilibria of system (\ref{ModelWithHMC}). 
Due to the cascade structure of the model, the system is triangular: 
the dynamics of $(S_2,X_2)$ depend on $(S_1,X_1)$, whereas 
$(X_0,S_1,X_1)$ evolve independently of $(S_2,X_2)$. 
As a consequence, the Jacobian matrix at an equilibrium 
$E=(X_0,S_1,X_1,S_2,X_2)$ has a block lower triangular form:
%=============================================
\[
J_E =
\begin{pmatrix}
J_F & 0 \\
B   & C
\end{pmatrix},
\]
%=============================================
where $J_F$ corresponds to the subsystem $(X_0,S_1,X_1)$ 
and $C$ to the subsystem $(S_2,X_2)$. 
Since $J_E$ is block lower triangular, its eigenvalues are the union of 
those of $J_F$ and $C$. 
Therefore, an equilibrium $E$ is locally exponentially stable (LES) 
if and only if all eigenvalues of $J_F$ and $C$ have negative real parts. 
In the following, we analyze separately the stability of each type of equilibrium.
%=============================================
%=============================================
\subsubsection{Stability of equilibria with $X_1=0$}
%=============================================
%=============================================
The Jacobian matrix of the subsystem $(X_0,S_1,X_1)$ is
%=============================================
\[
J_F =
\begin{pmatrix}
-\alpha_0 D - \mu_0'(X_0) X_1 & 0 & -\mu_0(X_0) \\[1mm]
k_0 \mu_0'(X_0) X_1 & -D - k_1 \mu_1'(S_1) X_1 & k_0 \mu_0(X_0) - k_1 \mu_1(S_1) \\[1mm]
0 & \mu_1'(S_1) X_1 & \mu_1(S_1) - D_1
\end{pmatrix},
\]
%=============================================
and that of the subsystem $(S_2,X_2)$ is
%=============================================
\[
C =
\begin{pmatrix}
- D - k_3 \mu_2'(S_2) X_2 & -k_3 \mu_2(S_2) \\[1mm]
\mu_2'(S_2) X_2 & \mu_2(S_2) - D_2
\end{pmatrix}.
\]
%=============================================
At the washout equilibrium 
$E_0^0 = \big(X_0^{in}/\alpha_0,\, S_1^{in},\, 0,\, S_2^{in},\, 0\big)$,
we obtain
\[
J_F(E_0^0) \!= \!
\begin{pmatrix}
-\alpha_0 D & 0 & -\mu_0(\frac{X_0^{in}}{\alpha_0}) \\[1mm]
0 & -D & k_0 \mu_0(\frac{X_0^{in}}{\alpha_0}) - k_1 \mu_1(S_1^{in}) \\[1mm]
0 & 0 & \mu_1(S_1^{in}) - D_1
\end{pmatrix},
C(E_0^0)  \!= \!
\begin{pmatrix}
- D & -k_3 \mu_2(S_2^{in}) \\[1mm]
0 & \mu_2(S_2^{in}) - D_2
\end{pmatrix}
\]
Since the eigenvalues are on the diagonal, the equilibrium $E_0^0$ is LES if and only if
\[
\mu_1(S_1^{in}) < D_1 \quad \text{and} \quad \mu_2(S_2^{in}) < D_2.
\]
Hence, we obtain the two conditions in Table \ref{TableExisStabWithMC}.
%=============================================

For $E_0^i$, we have $X_1=0$, $S_1=S_1^{in}$ and $S_2 = \lambda_2^i(D_2)$.  
Setting $X_2^*$ as the corresponding positive concentration, the matrices become
\[
J_F(E_0^i) = J_F(E_0^0), \quad
C(E_0^i) =
\begin{pmatrix}
- D - k_3 \mu_2'(\lambda_2^i) X_2^* & - k_3 \mu_2(\lambda_2^i) \\[1mm]
\mu_2'(\lambda_2^i) X_2^* & 0
\end{pmatrix}.
\]
The trace and determinant of $C(E_0^i)$ are
\[
\operatorname{tr}(C) = - D - k_3 \mu_2'(\lambda_2^i) X_2^* < 0, 
\quad
\det(C) = k_3 \mu_2(\lambda_2^i) \mu_2'(\lambda_2^i) X_2^*.
\]
Since $\mu_2(\lambda_2^i)>0$ and $X_2^*>0$, the sign of the determinant is determined by $\mu_2'(\lambda_2^i)$.  
Therefore, the equilibrium $E_0^i$ is LES if and only if
\[
\mu_1(S_1^{in}) < D_1 \quad \text{and} \quad \mu_2'(\lambda_2^i) > 0,
\]
that is, for $i=1$. Hence, we obtain the stability conditions in Table \ref{TableExisStabWithMC}.
%=============================================
%=============================================
\subsubsection{Stability of equilibria with $X_1>0$}
%=============================================
%=============================================
We consider equilibria corresponding to solutions $X_0^k$ of the equation
$\xi(X_0)=\delta(X_0)$, with associated quantities
\[
X_1^{k*}=\xi(X_0^k), \qquad S_1=\lambda_1(D_1).
\]
The Jacobian matrix at the equilibrium $F_1^{k}=(X_0^k,\lambda_1(D_1),X_1^{k*})$ is
\[
J_F(F_1^{k}) =
\begin{pmatrix}
-m_{11} & 0 & -m_{13} \\
 m_{21} & -m_{22} & \theta \\
 0      & m_{32}  & 0
\end{pmatrix},
\]
with
\[
\begin{aligned}
m_{11} &= \alpha_0 D + \mu_0'(X_0^k) X_1^{k*}, 
& m_{13} &= \mu_0(X_0^k), \\
m_{21} &= k_0 \mu_0'(X_0^k) X_1^{k*}, 
& m_{22} &= D + k_1 \mu_1'(\lambda_1) X_1^{k*}, \\
\theta &= k_0 \mu_0(X_0^k) - k_1 D_1, 
& m_{32} &= \mu_1'(\lambda_1) X_1^{k*}.
\end{aligned}
\]
All coefficients $m_{11}, m_{13}, m_{21}, m_{22}, m_{32}$ are positive, while $\theta$ may have any sign.  
The characteristic polynomial is
\[
P_{J_F}(\lambda)=\lambda^3 + c_1 \lambda^2 + c_2 \lambda + c_3,
\]
with
\[
c_1 = m_{11}+m_{22}, \quad 
c_2 = m_{11}m_{22} - \theta m_{32}, \quad 
c_3 = m_{32}(m_{21}m_{13}-\theta m_{11}).
\]
A straightforward calculation yields
\[
c_3 = m_{32} \Big[ - \alpha_0 D \, k_0 \mu_0(X_0^k) + k_1 D_1 m_{11} \Big].
\]
Using $\xi(X_0^k)=X_1^{k*}$ and the expression of $\xi'$, we obtain
\[
\xi'(X_0^k) = -\frac{m_{11}}{\mu_0(X_0^k)},
\]
so that
\[
c_3 = - k_1 D_1 \mu_0(X_0^k) m_{32}
\left(
\xi'(X_0^k) + \tfrac{\alpha_0 D k_0}{k_1 D_1}
\right).
\]
Hence, since $m_{32}>0$ and $\mu_0(X_0^k)>0$, the sign of $c_3$ is determined by
\[
\xi'(X_0^k) + \tfrac{\alpha_0 D k_0}{k_1 D_1}.
\]
By the Routh--Hurwitz criterion, a sufficient condition for $F_1^k$ to be LES is
\[
c_3>0 \quad \text{and} \quad c_4:=c_1c_2-c_3>0.
\]
Using $m_{11} = \alpha_0 D + \mu_0'(X_0^k) X_1^{k*}$, we obtain after straightforward computations
%=============================================
\begin{equation}                 \label{ExpC4}
c_4 =
k_1 m_{11}m_{32}  (\alpha_0 D - D_1) + \mu_0'(X_0^k) X_1^{k*} c_2
+ \alpha_0 D \, P + m_{22} c_2,
\end{equation}
%=============================================
where
\[
P = D m_{11} + k_1 D_1 m_{32} > 0.
\]
 \medskip
It remains to analyze the contribution of the matrix $C$.

\medskip
\noindent
\textbf{Case $E_1^{0k}$ ($X_2=0$).}
At this equilibrium, we have
\[
S_2^{k*} = S_2^{in} + k_2 \frac{D_1}{D} X_1^{k*}.
\]
The matrix $C$ is triangular, with eigenvalues $-D$ and 
$\mu_2(S_2^{k*}) - D_2$. Hence, both eigenvalues of $C$ have negative real parts if and only if
\[
S_2^{k*} \notin 
\left[\lambda_2^1(D_2),\,\lambda_2^2(D_2)\right].
\]
\medskip
\noindent
\textbf{Case $E_1^{ik}$ ($X_2>0$).}
At this equilibrium, $S_2=\lambda_2^i(D_2)$ and $X_2^{ik*}>0$. 
The determinant and trace of $C$ are given respectively by
\[
\det(C) = k_3 D_2\, \mu_2'(\lambda_2^i(D_2))\, X_2^{ik*}, 
\quad 
\operatorname{tr}(C) = -D - k_3 \mu_2'(\lambda_2^i(D_2)) X_2^{ik*}.
\]
Hence, the eigenvalues of $C$ have negative real parts if and only if
\[
\det(C)>0 \quad \text{and} \quad \operatorname{tr}(C)<0.
\]
Since $D_2>0$ and $X_2^{ik*}>0$, we have $\det(C)>0$ if and only if
$
\mu_2'(\lambda_2^i(D_2)) > 0,
$
that is, if $i=1$. In this case, $\operatorname{tr}(C)<0$, and thus both eigenvalues have negative real parts.

Combining these conditions with those obtained for $J_F$ yields the stability conditions of the equilibria $E_1^{0k}$ and $E_1^{ik}$, summarized in Table \ref{TableExisStabWithMC}.
%=============================================
%============================================= 
%%%%%%%%%%%%%%%%%%%%%%%%%%%%%%%%%%%%%%%%%%%%%%%%%%%%%%%%%%%%%%%%%%%%%%
%%%%%%%%%%%%%%%%%%%%%%%%%%%%%%%%%%%%%%%%%%%%%%%%%%%%%%%%%%%%%%%%%%%%%%
\section*{Acknowledgments}
The author acknowledges the support of the Euro-Mediterranean research network \href{https://treasure.hub.inrae.fr/}{Treasure}
 and the Tunisian Ministry of Higher Education and Scientific Research (Young Researchers' Encouragement Program: 06P1D2024-PEJC).
%%%%%%%%%%%%%%%%%%%%%%%%%%%%%%%%%%%%%%%%%%%%%%%%%%%%%%%%%%%%%%%%%%%%%%
%%%%%%%%%%%%%%%%%%%%%%%%%%%%%%%%%%%%%%%%%%%%%%%%%%%%%%%%%%%%%%%%%%%%%%
% \begin{credits}
% \subsubsection{\discintname}
% \textbf{The author has no competing interests to declare that are relevant to the content of this article.}
% \end{credits}
%%%%%%%%%%%%%%%%%%%%%%%%%%%%%%%%%%%%%%%%%%%%%%%%%%%%%%%%%%%%%%%%%%%%%%
%%%%%%%%%%%%%%%%%%%%%%%%%%%%%%%%%%%%%%%%%%%%%%%%%%%%%%%%%%%%%%%%%%%%%%---- Bibliography ----

%==============================================
\end{document}